\documentclass[10pt, reqno]{amsart}
\usepackage{amssymb}

\usepackage{hyperref}

\newcommand{\qtq}[1]{\quad\text{#1}\quad}

\newcommand{\R}{\mathbb{R}}

\newcommand{\C}{\mathbb{C}}

\newtheorem{theorem}{Theorem}[section]

\newtheorem{lemma}[theorem]{Lemma}

\newtheorem{proposition}[theorem]{Proposition}

\theoremstyle{definition}
\newtheorem{definition}[theorem]{Definition}

\theoremstyle{remark}

\numberwithin{equation}{section}

\begin{document}

\title[Dispersive decay for mass-critical NLS]{Dispersive decay for the mass-critical\\nonlinear Schr\"odinger equation}

\author[C. Fan]{Chenjie Fan}
\address{Academy of Mathematics and Systems Science
and
Hua Loo-Keng Key Laboratory of Mathematics, Chinese Academy of Sciences,
Beijing, China.}
\email{fancj@amss.ac.cn}

\author[R. Killip]{Rowan Killip}
\address{Department of Mathematics,  UCLA, Los Angeles, USA.}
\email{killip@math.ucla.edu}

\author[M. Visan]{Monica Visan}
\address{Department of Mathematics,  UCLA, Los Angeles, USA.}
\email{visan@math.ucla.edu}

\author[Z. Zhao]{Zehua Zhao}
\address{Department of Mathematics and Statistics, Beijing Institute of Technology
and
MIIT Key Laboratory of Mathematical Theory and Computation in Information Security, Beijing, China.}
\email{zzh@bit.edu.cn}

\maketitle

\begin{abstract}  We prove dispersive decay, pointwise in time, for solutions to the mass-critical nonlinear Schr\"odinger equation in spatial dimensions $d=1,2,3$.  
\end{abstract}

\section{Introduction}

We consider the long-time behavior of solutions to the initial-value problem 
\begin{equation}\label{NLS}
\begin{cases}
i\partial_t u+ \Delta u = \pm |u|^{\frac4d}u\\
u(0,x) = u_0(x) \in L^2(\R^d).
\end{cases}
\end{equation}
With the plus sign, this is a defocusing nonlinear Schr\"odinger equation; with the minus sign, it is focusing.
Our particular choice of nonlinearity makes the equation \emph{mass-critical}: the scaling symmetry associated with \eqref{NLS}, namely, 
\begin{align*}
u(t,x)\mapsto u_\lambda(t,x):= \lambda^{\frac d2} u(\lambda^2t, \lambda x), \qquad \lambda>0,
\end{align*}
also leaves invariant the (conserved) \emph{mass} of solutions
\begin{align*}
M(u) = \int_{\R^d} |u(t,x)|^2\, dx.
\end{align*}

In an impressive series of works \cite{Dodson1, Dodson2, Dodson3, Dodson4}, Dodson proved global well-posedness and scattering for solutions to \eqref{NLS}.  In the focusing case, this holds under the additional assumption that the mass of the initial data is smaller than that of the ground state $Q$, which is the unique non-negative, radial, $H^1(\R^d)$-solution to the elliptic equation
$$
\Delta Q + |Q|^{\frac4d}Q = Q.
$$
This mass restriction is sharp: finite time blowup can occur for solutions to \eqref{NLS} with initial data with mass $M(u_0)\geq M(Q)$.  Our next theorem summarizes the results of Dodson from \cite{Dodson1, Dodson2, Dodson3, Dodson4}; for earlier work on this problem, see \cite{KTV,KV-Clay,KVZ,TVZ1,TVZ2}.

\begin{theorem}[Dodson]\label{T:Dodson}
Fix $d\geq 1$ and let $u_0\in L^2(\R^d)$.  In the focusing case assume also that $M(u_0)<M(Q)$.  Then there exists a unique global solution $u\in (C_tL_x^2\cap L_{t,x}^{2(d+2)/d})(\R\times\R^d)$ to \eqref{NLS} and
\begin{align}\label{S bdd}
\|u\|_{ L_{t,x}^{\frac{2(d+2)}d}(\R\times\R^d)} \leq C\bigl(\|u_0\|_{L^2_x}\bigr).
\end{align}
Moreover, there exist asymptotic states $u_\pm\in L^2(\R^d)$ such that
\begin{align}\label{scat}
\lim_{t\to \pm \infty} \|u(t)-e^{it\Delta}u_{\pm}\|_{L_x^2} = 0.
\end{align}
\end{theorem}

The assertion \eqref{scat} is known as \emph{scattering}.  It shows that the long-time behavior of solutions to the nonlinear equation \eqref{NLS} parallels that of the linear Schr\"odinger equation, at least in the $L_t^\infty L_x^2$ metric.

The bound \eqref{S bdd} is also true for the linear Schr\"odinger equation; indeed, this is a prototypical Strichartz inequality.  The integrability in time here expresses important physics: it shows that the solutions of the (non)linear Schr\"odinger equation disperse.  The dispersion encoded by \eqref{S bdd} weakens the effect of the nonlinearity thus providing a simple explanation (both intuitively and rigorously) for \eqref{scat}.

Among the many Schr\"odinger-admissible spacetime norms $L^p_t L^q_x$, we use the symmetric pair ($p=q$) in \eqref{S bdd} for simplicity.  Control on the other scaling-critical norms can then be deduced by well-known arguments; see, for example, the proof of \eqref{1d S intermediary} below.

Prior to the advent of Strichartz estimates, our understanding of the long-time behavior of solutions relied on a different expression of dispersion, namely, quantitative decay pointwise in time.  For the linear Schr\"odinger equation, this is exemplified by the classical dispersive estimate
\begin{align}\label{L disp est'}
\|e^{it\Delta}u_0\|_{L^{r}_x(\R^d)}\lesssim |t|^{-d(\frac12-\frac1r)} \|u_0\|_{L^{r'}_x} \quad\text{for $t\neq 0$ and $2\leq r\leq\infty$}.
\end{align}
The requirement that $u_0\in L^{r'}$ constitutes a spatial concentration requirement for the initial data.  This is what breaks the time translation symmetry and so makes a quantitative decay estimate possible.  See \cite{MR4576318} for further discussion of this point in the nonlinear setting.

In this paper, we  prove analogues of \eqref{L disp est'} for the nonlinear evolution \eqref{NLS}.  Our three theorems treat spatial dimensions one, two, and three, respectively.   
\begin{theorem}\label{T:1d}
Fix $2<r\leq\infty$.  Given $u_0\in L^2(\R)\cap L^{r'}(\R)$, let $u$ denote the global solution to \eqref{NLS} guaranteed by Theorem~\ref{T:Dodson}.  Then
\begin{align}\label{decay 1d}
\sup_{t\neq 0}\, |t|^{\frac12-\frac1r}\|u(t)\|_{L^r_x} \leq C\bigl(\|u_0\|_{L^2_x}\bigr) \|u_0\|_{L_x^{r'}} .
\end{align}
\end{theorem}

\begin{theorem}\label{T:2d}
Fix $2<r<\infty$.  Given $u_0\in L^2(\R^2)\cap L^{r'}(\R^2)$, let $u$ denote the global solution to \eqref{NLS} guaranteed by Theorem~\ref{T:Dodson}.  Then
\begin{align}\label{decay 2d}
\sup_{t\neq 0}\, |t|^{2(\frac12-\frac1r)}\|u(t)\|_{L^r_x} \leq C\bigl(\|u_0\|_{L^2_x}\bigr) \|u_0\|_{L_x^{r'}} .
\end{align}
\end{theorem}

\begin{theorem}\label{T:3d}
Fix $2<r<6$.  Given $u_0\in L^2(\R^3)\cap L^{r'}(\R^3)$, let $u$ denote the global solution to \eqref{NLS} guaranteed by Theorem~\ref{T:Dodson}.  Then 
\begin{align}\label{decay 3d}
\sup_{t\neq 0}\, |t|^{3(\frac12-\frac1r)}\|u(t)\|_{L^r_x} \leq C\bigl(\|u_0\|_{L^2_x}\bigr) \|u_0\|_{L_x^{r'}}.
\end{align}
\end{theorem}

We feel that these estimates embody the ideal nonlinear version of \eqref{L disp est'}, at least for the range of parameters treated: We make no auxiliary assumptions besides finiteness of the critical norm, which is key for the existence of solutions.  Moreover, the non-critical norm appears linearly on the right-hand side, just as in the linear estimate \eqref{L disp est'}.  As we will see, these two features distinguish our results from prior works in this direction.

Our results in one and two spatial dimensions will be proved by parallel methods, using crucially Lorentz-space improvements of the traditional Strichartz inequality; see Proposition~\ref{P:Lorentz Strichartz}.  The two-dimensional case is treated in Section~\ref{S:3}, which serves as a model for the one-dimensional case in Section~\ref{S:4}.  In three dimensions, however, this approach is inapplicable due to the low power of the nonlinearity.  Our argument in the three-dimensional case is presented in Section~\ref{S:5}.  It relies on a subtle decomposition of the nonlinearity and the unusual linear estimates presented in Lemma~\ref{L:3adj}.

Neither of the methods presented here appear to be capable of demonstrating decay that is $O(|t|^{-1})$ or faster without additional assumptions on the initial data beyond $u_0\in L^2_x\cap L^{r'}_x$.  This underlies the exclusion of $L^\infty_x$ decay in two space dimensions and the limitation $r<6$ in three dimensions.

\subsection*{Prior work} As mentioned earlier, quantitative decay pointwise in time has a long history in dispersive PDE.  Initially, this was achieved only under strong regularity and decay hypotheses; see, for example, \cite{MR0877907,MR0784477,MR0680085,MR0515228,MR0303097,MR0733654} as well as the references therein.

Subsequent progress has significantly reduced the regularity and/or decay properties needed in order to construct local solutions.  The earlier role of pointwise in time bounds has been overtaken by $L^q_t L^r_x$ bounds of the type \eqref{S bdd}, which capture both the local-in-time regularity needed to construct solutions and the long-time decay needed to prove scattering.  The time is thus ripe for understanding dispersive decay in the sense of the classical dispersive estimate \eqref{L disp est'} under minimal hypotheses.  This gauntlet has been taken up by several different authors \cite{FSZ,MR4251840,MR4576877,Looi}.

The paper \cite{FSZ} focuses on the cubic nonlinear Sch\"odinger equation in $\R^3$ with initial data in $H^1$.  Although not scaling-critical, mass and energy conservation make this a physically natural class of initial data.  The papers \cite{MR4251840,MR4576877} focus on the energy-critical nonlinear Schr\"odinger equation, both requiring initial data in $H^3(\R^3)$ in order to demonstrate full $L^\infty_x$ dispersive decay.  None of these papers demonstrate linear dependence on $\| u_0\|_{L^{r'}_x}$.

The paper \cite{Looi} studies the energy-critical wave equation under significantly stronger regularity and decay hypotheses; however, this paper considers wave equations associated to rather general spacetime metrics in connection with the problem of describing late-time tails in general relativity.

Based on the available evidence, it is natural to imagine that an analogue of our theorems should hold for any nonlinear wave or Schr\"odinger equation that admits scaling-critical spacetime bounds strong enough to imply scattering.  We are not aware of any obstructions to this notion and hope this paper will stimulate further activity on this question. 


\subsection*{Acknowledgements} C. F. was partially supported by the National Key R\&D Program of China, 2021YFA1000800, CAS Project for Young Scientists in Basic Research
Grant No. YSBR-031, NSFC grant No. 12288201, and a fund from AMSS. R.K. was supported by U.S. NSF grant DMS-2154022. M.V. was supported by U.S. NSF grant DMS-2054194.  Z. Z. was supported by the NSF grant of China (No. 12101046, 12271032) and the Beijing Institute of Technology Research Fund Program for Young Scholars.

\section{Notation and useful lemmas}\label{S:2}

We write $A\lesssim B$ to indicate $A\leq CB$ with an absolute constant $C$ whose specific value is of no consequence.  We write $A\sim B$ when $A\lesssim B$ and $B\lesssim A$.  The maximum and minimum of two numbers $s$ and $t$ are abbreviated $s\vee t$ and $s\wedge t$, respectively.

In our treatment of the one- and two-dimensional Schr\"odinger equations, we will employ refinements of the dispersive estimates in Lorentz spaces.  We review the definitions and basic tools below; for a textbook treatment of Lorentz spaces, see~\cite{CFA}. 

\begin{definition}[Lorentz spaces]
Fix $d\geq 1$, $1\leq p<\infty$, and $1\leq q\leq \infty$.  The Lorentz space $L^{p,q}(\R^d)$ is the space of measurable functions $f:\R^d\to \C$ for which the quasinorm
$$
\|f\|_{L^{p,q}(\R^d)}^*=p^{\frac1q}\bigl\| \lambda \bigl| \{x\in \R^d:\, |f(x)|>\lambda\}\bigr|^{\frac1p}\bigr\|_{L^q((0, \infty), \frac{d\lambda}{\lambda})}
$$
is finite.  Here, $|A|$ denotes the Lebesgue measure of the set $A\subseteq\R^d$.
\end{definition}

Note that $L^{p,p}(\R^d)$ coincides with the standard Lebesgue space $L^p(\R^d)$, while $L^{p,\infty}(\R^d)$ coincides with the weak $L^p(\R^d)$ space.  

Lorentz spaces satisfy the following nesting property: $L^{p,q}(\R^d)\subset L^{p,r}(\R^d)$ whenever $1\leq p<\infty$ and $1\leq q<r\leq \infty$.  Indeed, there exists a constant $C(p,q,r)$ so that
\begin{align}\label{nesting}
\big\|f\big\|_{L^{p,r}(\R^d)}\leq C(p,q,r)\big\|f\big\|_{L^{p,q}(\R^d)}.
\end{align}

H\"older's inequality carries over to Lorentz spaces in the following form:

\begin{lemma}[H\"older inequality in Lorentz spaces]\label{L:Holder}
Given $1\leq p, p_1, p_2<\infty$ and $1\leq q, q_1, q_2\leq \infty$ such that $\frac1p=\frac1{p_1} + \frac1{p_2}$ and $\frac1q=\frac1{q_1} + \frac1{q_2}$, we have
\begin{align*}
\|fg\|_{L^{p,q}(\R^d)} \lesssim \|f\|_{L^{p_1,q_1}(\R^d)}\|g\|_{L^{p_2,q_2}(\R^d)},
\end{align*}
with the implicit constant depending only on the parameters $d, p, p_1, p_2, q, q_1, q_2$.
\end{lemma}

Using the Hunt Interpolation Theorem \cite{Hunt64, Hunt66} to interpolate between the dispersive estimate for the Schr\"odinger propagator
\begin{align}\label{E dis}
\|e^{it\Delta}u_0\|_{L^\infty_x(\R^d)}\lesssim |t|^{-\frac d2} \|u_0\|_{L^1_x} \quad\text{for $t\neq 0$}
\end{align} 
and the conservation of mass
$$
\|e^{it\Delta}u_0\|_{L^2_x(\R^d)}=\|u_0\|_{L^2_x} \quad\text{for $t\in \R$},
$$ 
yields the following Lorentz improvement of the standard dispersive estimates: for any $2<r<\infty$ and $1\leq \theta\leq \infty$,
\begin{align}\label{L disp est}
\|e^{it\Delta}u_0\|_{L^{r, \theta}_x(\R^d)}\lesssim |t|^{-d(\frac12-\frac1r)} \|u_0\|_{L^{r', \theta}_x} \quad\text{for $t\neq 0$}.
\end{align}
Note that by taking $\theta=r>r'$ and invoking \eqref{nesting}, we recover the standard dispersive estimates \eqref{L disp est'}.

As observed in \cite{MR1856248}, applying the Hardy--Littlewood--Sobolev inequality in Lorentz spaces in the standard proof of the Strichartz estimates yields the following Lorentz-space improvement:

\begin{proposition}[Lorentz--Strichartz estimates]\label{P:Lorentz Strichartz}
Fix $d\geq 1$ and let $(p,q)$ be a Schr\"odinger-admissible pair, that is, $2\leq p,q\leq \infty$, $\frac2p+ \frac dq=\frac d2$, and $(d,p,q)\neq (2,2,\infty)$.  Then
\begin{align*}
\|e^{it\Delta}u_0\|_{L_t^{p,2} L^q_x(\R\times\R^d)} &\lesssim \|u_0\|_{L^2_x(\R^d)},\\
\Bigl\| \int_0^t e^{i(t-s)\Delta} F(s)\, ds\Bigr\|_{L_t^{p,2} L^q_x(\R\times\R^d)}&\lesssim \|F\|_{L_t^{p',2} L_x^{q'}(\R\times\R^d)}.
\end{align*}
\end{proposition}

The following two lemmas record spacetime bounds in Lorentz spaces for solutions to \eqref{NLS}.  As we will show below, these follow from the global spacetime bounds in Theorem~\ref{T:Dodson} and Proposition~\ref{P:Lorentz Strichartz}.  While the arguments we present can be adapted to derive Lorentz spacetime bounds for solutions to \eqref{NLS} in all dimensions $d\geq 1$, we only present the details in the settings that are relevant to our applications.  Specifically, these Lorentz improvements constitute a key ingredient in the proofs of Theorems~\ref{T:1d} and \ref{T:2d} and so we present the details for $d=1,2$ only. 

\begin{lemma}[Lorentz spacetime bounds $d=1$]\label{L:L 1d}
Given $u_0\in L^2(\R)$ satisfying the hypotheses of Theorem~\ref{T:Dodson}, let $u$ denote the unique global solution to \eqref{NLS} guaranteed by that theorem.  For any $2<r\leq \infty$ we have the global spacetime bound
\begin{align}\label{Lorentz bdd 1d}
\|u\|_{L_t^{\frac{8r}{r+2},4}L_x^{\frac{4r}{r-2}}(\R\times\R)}\lesssim \|u\|_{L_t^{\frac{8r}{r+2},2}L_x^{\frac{4r}{r-2}}(\R\times\R)} \leq C(\|u_0\|_{L^2_x}).
\end{align}
\end{lemma}

\begin{proof}
The first inequality in \eqref{Lorentz bdd 1d} follows immediately from \eqref{nesting}.  We thus turn to the second inequality in \eqref{Lorentz bdd 1d}.  In the computations that follow we will use the fact that for $2<r\leq\infty$, the pairs $(\frac{8r}{r+2},\frac{4r}{r-2})$, $(\frac{16r}{3r-2},\frac{8r}{r+2})$, and $(\frac{48r}{7r+6}, \frac{24r}{5r-6})$ are Schr\"odinger-admissible pairs. 

Using the Duhamel formula
$$
u(t) = e^{it\Delta}u_0 - i \int_0^t e^{i(t-s)\Delta} (|u|^4u)(s)\, ds
$$
together with the Strichartz and H\"older inequalities, we may estimate
\begin{align}\label{1d S}
\|u\|_{L_t^{\frac{16r}{3r-2}}L_x^{\frac{8r}{r+2}}(\R\times\R)} &\lesssim \|u_0\|_{L_x^2} + \||u|^4u\|_{L_t^{\frac{48r}{41r-6}}L_x^{\frac{24r}{19r+6}}(\R\times\R)}\notag\\
&\lesssim \|u_0\|_{L_x^2} +\|u\|_{L_t^{\frac{16r}{3r-2}}L_x^{\frac{8r}{r+2}}(\R\times\R)} \|u\|_{L_{t,x}^6(\R\times\R)}^4.
\end{align}
Using \eqref{S bdd} and subdividing $\R$ into finitely many subintervals $I_j$ on which 
$$
\|u\|_{L_{t,x}^6(I_j\times\R)}<\eta
$$
for $\eta>0$ small (depending only on the implicit constant in \eqref{1d S}), a standard bootstrap argument yields
\begin{align}\label{1d S intermediary}
\|u\|_{L_t^{\frac{16r}{3r-2}}L_x^{\frac{8r}{r+2}}(\R\times\R)}\leq C(\|u_0\|_{L_x^2}).
\end{align}

Applying Proposition~\ref{P:Lorentz Strichartz}, we may thus bound
\begin{align}\label{1d S'}
\|u\|_{L_t^{\frac{8r}{r+2},2}L_x^{\frac{4r}{r-2}}(\R\times\R)}
&\lesssim \|u_0\|_{L_x^2} + \||u|^4u\|_{L_t^{\frac{8r}{7r-2},2}L_x^{\frac{4r}{3r+2}}(\R\times\R)}\notag\\
&\lesssim \|u_0\|_{L_x^2} +\|u\|_{L_t^{\frac{8r}{r+2},2}L_x^{\frac{4r}{r-2}}(\R\times\R)}\|u\|_{L_t^{\frac{16r}{3r-2}}L_x^{\frac{8r}{r+2}}(\R\times\R)}^4.
\end{align}
Using \eqref{1d S intermediary} and subdividing $\R$ into finitely many subintervals $I_j$ on which 
$$
\|u\|_{L_t^{\frac{16r}{3r-2}}L_x^{\frac{8r}{r+2}}(I_j\times\R)}<\eta
$$
for $\eta>0$ small (depending only on the implicit constant in \eqref{1d S'}), we get \eqref{Lorentz bdd 1d}.
\end{proof}

\begin{lemma}[Lorentz spacetime bounds $d=2$]\label{L:L 2d}
Given $u_0\in L^2(\R^2)$ satisfying the hypotheses of Theorem~\ref{T:Dodson}, let $u$ denote the unique global solution to \eqref{NLS} guaranteed by that theorem.  For any $2<r<\infty$ we have the global spacetime bound
\begin{align}\label{Lorentz bdd}
\|u\|_{L_t^{r,2}L_x^{\frac{2r}{r-2}}(\R\times\R^2)} \leq C(\|u_0\|_{L_x^2}).
\end{align}
\end{lemma}

\begin{proof}
In two spatial dimensions, both $(r, \frac{2r}{r-2})$ and $(\frac{2r}{r-2},r)$ constitute Schr\"odinger-admissible pairs for any $2<r<\infty$.  The proof parallels that of Lemma~\ref{L:L 1d}.  The role of \eqref{1d S} is played by bound
\begin{align*}
\|u\|_{L_t^{\frac{2r}{r-2}}L_x^r(\R\times\R^2)} &\lesssim \|u_0\|_{L_x^2} + \||u|^2u\|_{L_t^{\frac{r}{r-1}}L_x^{\frac{2r}{r+2}}(\R\times\R^2)}\\
&\lesssim \|u_0\|_{L_x^2} + \|u\|_{L_t^{\frac{2r}{r-2}}L_x^r(\R\times\R^2)} \|u\|_{L_{t,x}^4(\R\times\R^2)}^2,
\end{align*}
while the role of \eqref{1d S'} is played by the bound
\begin{align*}
\|u\|_{L_t^{r,2}L_x^{\frac{2r}{r-2}}(\R\times\R^2)} &\lesssim \|u_0\|_{L_x^2} + \||u|^2u\|_{L_t^{\frac{r}{r-1},2}L_x^{\frac{2r}{r+2}}(\R\times\R^2)}\\
&\lesssim \|u_0\|_{L_x^2} + \|u\|_{L_t^{r,2}L_x^{\frac{2r}{r-2}}(\R\times\R^2)} \|u\|_{L_t^{\frac{2r}{r-2}}L_x^r(\R\times\R^2)}^2.\qedhere
\end{align*}
\end{proof}

To prove Theorem~\ref{T:3d} we will employ the following linear estimate to control the contribution of the nonlinearity:

\begin{lemma}\label{L:3adj}
Fix $2<r<6$.  For any $t\in \R$ and any interval $I\subseteq [0,t]$ we have 
\begin{align}\label{improv bdd}
&\Bigl\| \int_I e^{i(t-s)\Delta} F(s)\, ds\Bigr\|_{L^r_x(\R^3)}\\
&\lesssim \min\Bigl\{ \bigl\||t-s|^{-3(\frac12-\frac1r)} F(s)\bigr\|_{L_s^{1}L_x^{\frac{r}{r-1}}(I\times\R^3)},\bigl\||t-s|^{-3(\frac12-\frac1r)} F(s)\bigr\|_{L_s^{\frac{r}{r-1}}L_x^{\frac{3r}{3r-1}}(I\times\R^3)} \Bigr\}.\notag
\end{align}
\end{lemma}

\begin{proof}
The first bound on the right-hand side of \eqref{improv bdd} follows immediately from the dispersive estimate \eqref{L disp est'}:
\begin{align*}
\Bigl\| \int_I e^{i(t-s)\Delta} F(s)\, ds\Bigr\|_{L^r_x} \lesssim \int_I |t-s|^{-3(\frac12-\frac1r)} \|F(s)\|_{L_x^{\frac{r}{r-1}}}\, ds.
\end{align*}

The second bound may be obtained by complex interpolation between the end-point Strichartz estimate
\begin{align*}
\Bigl\| \int_I e^{i(t-s)\Delta} F(s)\, ds\Bigr\|_{L^2_x} \lesssim \|F\|_{L_{s}^{2\vphantom/}L_{x}^{6/5}(I\times\R^3)}
\end{align*}
from \cite{KT} and the following consequence of the dispersive estimate \eqref{E dis}:
\begin{align*}
\Bigl\| \int_I e^{i(t-s)\Delta} F(s)\, ds\Bigr\|_{L^\infty_x} \lesssim \int_I |t-s|^{-\frac32} \|F(s)\|_{L_x^1}\, ds \lesssim \bigl\||t-s|^{-\frac32} F(s)\bigr\|_{L_{s,x}^1(I\times\R^3)}.
\end{align*}

This completes the proof of the lemma.
\end{proof}

\section{Proof of Theorem~\ref{T:2d}}\label{S:3}

By time-reversal symmetry, it suffices to prove the pointwise dispersive estimate \eqref{decay 2d} for $t\in (0, \infty)$.  Using the density of Schwartz functions in $L^2(\R^2)\cap L^{r'}(\R^2)$, it suffices to prove \eqref{decay 2d} for Schwartz solutions to \eqref{NLS}.  For $T\in (0, \infty]$ we define
\begin{align}\label{X def}
\|u\|_{X(T)} := \sup_{t\in [0, T)} t^{2(\frac12-\frac1r)} \|u(t)\|_{L^r_x}.
\end{align}

We employ a small parameter $\eta>0$ that will be chosen later, depending only on absolute constants (such as those in the dispersive and Strichartz estimates).  Using \eqref{Lorentz bdd}, we may decompose $[0, \infty)$ into $J = J(\|u_0\|_{L_x^2}, \eta)$ many intervals $I_j = [T_{j-1}, T_j]$ such that 
\begin{align}\label{small norm 2d}
 \|u\|_{L_t^{r,2}L_x^q(I_j\times\R^2)}< \eta \qtq{with} q:=\tfrac{2r}{r-2}
\end{align}
and $J$ chosen as small as possible.  Recall that  $(r,q)$ is a Schr\"odinger-admissible pair.

Our central goal is to show that for each $1\leq j\leq J$, we have
\begin{align}\label{X 2d bdd}
X(T_j) \lesssim \|u_0\|_{L_x^{r'}} + C(\|u_0\|_{L_x^2})X(T_{j-1}) +\eta^{2} X(T_j).
\end{align}
Choosing $\eta$ small enough to defeat the absolute implicit constant here, we obtain
$$
X(\infty)\leq C(\|u_0\|_{L_x^2})\|u_0\|_{L_x^{r'}},
$$
which proves \eqref{decay 2d}.

We turn to \eqref{X 2d bdd}.  Fix $1\leq j\leq J$. For $t\in(0, T_j)$, we decompose 
\begin{align}\label{2d decomp}
u(t) &= e^{it\Delta} u_0 - i\int_0^{\frac t2} e^{i(t-s)\Delta} F(s)\, ds-  i\int_{\frac t2}^t e^{i(t-s)\Delta} F(s)\, ds,
\end{align}
where $F= \pm |u|^2u$.  

By the dispersive estimate \eqref{L disp est'}, the contribution of the first term on RHS\eqref{2d decomp} is easily seen to be acceptable:
\begin{align}\label{2d 1}
\|e^{it\Delta} u_0\|_{L^r_x} \lesssim t^{-2(\frac12-\frac1r)} \| u_0\|_{L_x^{r'}}.
\end{align}

To estimate the contribution of the second term on RHS\eqref{2d decomp}, we decompose the region of integration into $[0,\frac t2] \cap [0, T_{j-1}]$ and $[0,\frac t2] \cap [T_{j-1}, T_j]$.  On either of these regions we have that $|t-s|\sim t$.  Using the dispersive estimate \eqref{L disp est'}, followed by the H\"older inequality from Lemma~\ref{L:Holder} and then \eqref{Lorentz bdd}, we obtain
\begin{align}\label{2d 2}
\Bigl\| \int_{[0,\frac t2] \cap [0, T_{j-1}]} & e^{i(t-s)\Delta}  F(s)\, ds \Bigr\|_{L_x^r}\notag\\
&\lesssim  \int_0^{\frac t2 \wedge T_{j-1}} |t-s|^{-2(\frac12-\frac1r)}\|F(s)\|_{L_x^{r'}}\, ds\notag\\
&\lesssim |t|^{-2(\frac12-\frac1r)} \int_0^{\frac t2 \wedge T_{j-1}} \|u(s)\|_{L_x^r} \|u(s)\|_{L_x^q}^2\, ds\notag\\
&\lesssim |t|^{-2(\frac12-\frac1r)} \int_0^{\frac t2 \wedge T_{j-1}}\!\! |s|^{-2(\frac12-\frac1r)} \|u\|_{X(T_{j-1})} \|u(s)\|_{L_x^q}^2\, ds\notag\\
&\lesssim |t|^{-2(\frac12-\frac1r)} \|u\|_{X(T_{j-1})} \|u\|_{L_s^{r,2}L_x^q ([0,\frac t2 \wedge T_{j-1}]\times\R^2)}^2\bigl\| |s|^{-\frac{r-2}{r}}\bigr\|_{L_x^{\frac{r}{r-2},\infty}}\notag\\
&\leq |t|^{-2(\frac12-\frac1r)} \|u\|_{X(T_{j-1})}C(\|u_0\|_{L_x^2}).
\end{align}
Observe that the Lorentz space improvement in \eqref{Lorentz bdd} was crucial to compensate for the fact that $|s|^{-2(\frac12-\frac1r)}$ lies only in the Lorentz space $L^{\frac r{r-2},\infty}_s(\R)$.

Arguing similarly, but using \eqref{small norm 2d} in place of \eqref{Lorentz bdd}, we get
\begin{align}\label{2d 2'}
\Bigl\| \int_{[0,\frac t2] \cap [T_{j-1}, T_j]} e^{i(t-s)\Delta} F(s)\, ds \Bigr\|_{L_x^r}
&\lesssim |t|^{-2(\frac12-\frac1r)} \|u\|_{X(T_j)} \|u\|_{L_t^{r,2}L_x^q (I_j\times\R^2)}^2\notag\\
&\lesssim |t|^{-2(\frac12-\frac1r)} \eta^2 \|u\|_{X(T_j)}.
\end{align}

Finally, we turn to the contribution of the third term on RHS\eqref{2d decomp}.  In this case, we observe that $s\in [\frac t2, t]$ ensures $s\sim t$.  We will further subdivide the domain of integration into $[\frac t2, t] \cap [0, T_{j-1}]$ and $[\frac t2, t] \cap [T_{j-1}, T_j]$.  Arguing as above, we obtain
\begin{align}\label{2d 3}
\Bigl\| \int_{[\frac t2, t] \cap [0, T_{j-1}]} & e^{i(t-s)\Delta} F(s)\, ds \Bigr\|_{L_x^r}\notag\\
&\lesssim  \int_{\frac t2}^{t\wedge T_{j-1}} |t-s|^{-2(\frac12-\frac1r)}\|F(s)\|_{L_x^{r'}}\, ds\notag\\
&\lesssim  \int_{\frac t2}^{t\wedge T_{j-1}} |t-s|^{-2(\frac12-\frac1r)}\|u(s)\|_{L_x^r} \|u(s)\|_{L_x^q}^2\, ds\notag\\
&\lesssim  \int_{\frac t2}^{t\wedge T_{j-1}} |t-s|^{-2(\frac12-\frac1r)} |s|^{-2(\frac12-\frac1r)} \|u\|_{X(T_{j-1})} \|u(s)\|_{L_x^q}^2\, ds\notag\\
&\lesssim |t|^{-2(\frac12-\frac1r)} \|u\|_{X(T_{j-1})} \|u\|_{L_s^{r,2}L_x^q ([\frac t2, t \wedge T_{j-1}]\times\R^2)}^2\bigl\| |t-s|^{-\frac{r-2}{r}}\bigr\|_{L_x^{\frac{r}{r-2},\infty}}\notag\\
&\leq |t|^{-2(\frac12-\frac1r)} \|u\|_{X(T_{j-1})}C(\|u_0\|_{L^2_x}).
\end{align}
Once again, the Lorentz space improvement in \eqref{Lorentz bdd} was used to compensate for the fact that $|t-s|^{-2(\frac12-\frac1r)}$ belongs only to the Lorentz space $L^{\frac r{r-2},\infty}_s(\R)$.

Arguing similarly, but using \eqref{small norm 2d} in place of \eqref{Lorentz bdd}, we get
\begin{align}\label{2d 3'}
\Bigl\| \int_{[\frac t2, t] \cap [T_{j-1}, T_j]} e^{i(t-s)\Delta} F(s)\, ds \Bigr\|_{L_x^r}
&\lesssim |t|^{-2(\frac12-\frac1r)} \|u\|_{X(T_j)} \|u\|_{L_t^{r,2}L_x^q (I_j\times\R^2)}^2\notag\\
&\lesssim |t|^{-2(\frac12-\frac1r)} \eta^2 \|u\|_{X(T_j)}.
\end{align}

Combining \eqref{2d decomp} through \eqref{2d 3'} we obtain \eqref{X 2d bdd}, which completes the proof of Theorem~\ref{T:2d}. \qed

\section{Proof of Theorem~\ref{T:1d}}\label{S:4}

The proof of Theorem~\ref{T:1d} parallels that of Theorem~\ref{T:2d}.  For $2<r\leq \infty$ fixed and a global Schwartz solution $u:\R\times\R\to\C$ to \eqref{NLS}, we define
$$
\|u\|_{X(T)}:= \sup_{t\in [0, T)} t^{\frac12-\frac1r} \|u(t)\|_{L^r_x} \quad\text{for $0<T\leq \infty$}.
$$

Using \eqref{Lorentz bdd 1d}, we may decompose $[0, \infty)$ into $J = J(\|u_0\|_{L^2_x}, \eta)$ many intervals $I_j = [T_{j-1}, T_j)$ such that 
\begin{align}\label{small norm 1d}
 \|u\|_{L_t^{\frac{8r}{r+2},4}L_x^{\frac{4r}{r-2}}(I_j\times\R)}< \eta
\end{align}
where $\eta>0$ will be chosen small, depending only on absolute constants.

Our goal is to show that for each $1\leq j\leq J$, we have
\begin{align}\label{X 1d bdd}
X(T_j) \lesssim \|u_0\|_{L_x^{r'}} + C(\|u_0\|_{L^2_x})X(T_{j-1}) +\eta^{4} X(T_j).
\end{align}
Choosing $\eta$ small enough to defeat the implicit constant in \eqref{X 1d bdd}, this yields
$$
X(\infty)\leq C(\|u_0\|_{L^2_x})\|u_0\|_{L_x^{r'}},
$$
which proves \eqref{decay 1d}.

We start with the analogues of \eqref{2d 2} and \eqref{2d 3} in the one-dimensional case.  Using the dispersive estimate \eqref{L disp est'}, the H\"older inequality, and \eqref{Lorentz bdd 1d}, we obtain
\begin{align*}
\Bigl\| \int_{[0,\frac t2] \cap [0, T_{j-1}]} &e^{i(t-s)\Delta}(|u|^4u)(s)\, ds \Bigr\|_{L_x^r}\\
&\lesssim  \int_0^{\frac t2 \wedge T_{j-1}} |t-s|^{-(\frac12-\frac1r)}\| (|u|^4u)(s)\|_{L_x^{r'}}\, ds\\
&\lesssim |t|^{-(\frac12-\frac1r)} \int_0^{\frac t2 \wedge T_{j-1}} \|u(s)\|_{L_x^r} \|u(s)\|_{L_x^{\frac{4r}{r-2}}}^4\, ds\\
&\lesssim |t|^{-(\frac12-\frac1r)} \int_0^{\frac t2 \wedge T_{j-1}}\!\! |s|^{-(\frac12-\frac1r)} \|u\|_{X(T_{j-1})}  \|u(s)\|_{L_x^{\frac{4r}{r-2}}}^4\, ds\\
&\lesssim |t|^{-(\frac12-\frac1r)} \|u\|_{X(T_{j-1})}\|u\|^4_{L_s^{\frac{8r}{r+2},4}L_x^{\frac{4r}{r-2}}([0,\frac t2 \wedge T_{j-1}]\times\R)}\bigl\| |s|^{-\frac{r-2}{2r}}\bigr\|_{L_x^{\frac{2r}{r-2},\infty}}\\
&\leq |t|^{-(\frac12-\frac1r)} \|u\|_{X(T_{j-1})}C(\|u_0\|_{L_x^2}).
\end{align*}
Observe that the Lorentz space improvement in \eqref{Lorentz bdd 1d} was needed to compensate for the fact that $|s|^{-(\frac12-\frac1r)}$ lies only in the Lorentz space $L^{\frac{2r}{r-2},\infty}_s(\R)$.  Similar considerations yield
\begin{align*}
\Bigl\| &\int_{[\frac t2, t] \cap [0, T_{j-1}]}  e^{i(t-s)\Delta}  (|u|^4u)(s)\, ds \Bigr\|_{L_x^r}\\
&\qquad\lesssim  \int_{\frac t2}^{t\wedge T_{j-1}} |t-s|^{-(\frac12-\frac1r)}\| (|u|^4u)(s)\|_{L_x^{r'}}\, ds\\
&\qquad\lesssim  \int_{\frac t2}^{t\wedge T_{j-1}} |t-s|^{-(\frac12-\frac1r)}\|u(s)\|_{L_x^r} \|u(s)\|_{L_x^{\frac{4r}{r-2}}}^4\, ds\\
&\qquad\lesssim  \int_{\frac t2}^{t\wedge T_{j-1}} |t-s|^{-(\frac12-\frac1r)} |s|^{-(\frac12-\frac1r)} \|u\|_{X(T_{j-1})} \|u(s)\|_{L_x^{\frac{4r}{r-2}}}^4\, ds\\
&\qquad\lesssim |t|^{-(\frac12-\frac1r)} \|u\|_{X(T_{j-1})} \|u\|^4_{L_s^{\frac{8r}{r+2},4}L_x^{\frac{4r}{r-2}}([\frac t2, t \wedge T_{j-1}]\times\R)}\bigl\| |t-s|^{-\frac{r-2}{2r}}\bigr\|_{L_x^{\frac{2r}{r-2},\infty}}\\
&\qquad\leq |t|^{-(\frac12-\frac1r)} \|u\|_{X(T_{j-1})}C(\|u_0\|_{L^2_x}).
\end{align*}

Using \eqref{small norm 1d} in place of \eqref{Lorentz bdd 1d} in the computations above, we get
\begin{align*}
\Bigl\| \int_{[0,\frac t2] \cap [T_{j-1}, T_j]} e^{i(t-s)\Delta}  (|u|^4u)(s)\, ds \Bigr\|_{L_x^r}
&\lesssim |t|^{-(\frac12-\frac1r)} \eta^4 \|u\|_{X(T_j)},\\
\Bigl\| \int_{[\frac t2, t] \cap [T_{j-1}, T_j]} e^{i(t-s)\Delta}  (|u|^4u)(s)\, ds \Bigr\|_{L_x^r}
&\lesssim |t|^{-(\frac12-\frac1r)} \eta^4 \|u\|_{X(T_j)}.
\end{align*}
Together with the dispersive estimate
\begin{align*}
\|e^{it\Delta} u_0\|_{L^r_x} \lesssim t^{-(\frac12-\frac1r)} \| u_0\|_{L_x^{r'}}
\end{align*}
and the Duhamel formula, this proves the desired bound \eqref{X 1d bdd}. \qed

\section{Proof of Theorem~\ref{T:3d}}\label{S:5}

By time-reversal symmetry, it suffices to prove the pointwise dispersive estimate \eqref{decay 3d} for $t\in (0, \infty)$.  Using the density of $H^1(\R^3)$ functions in $L^2(\R^3)\cap L^{r'}(\R^3)$, it suffices to prove \eqref{decay 3d} for $H^1(\R^3)$ solutions to \eqref{NLS}.  In particular, at each time $t\in (0,\infty)$, Sobolev embedding guarantees that $u(t) \in L^r(\R^3)$.  We may thus define
\begin{align}\label{X def}
\|u\|_{X(T)} := \sup_{t\in [0, T)} t^{3(\frac12-\frac1r)} \|u(t)\|_{L^r_x}
\end{align}
for each $T\in (0, \infty]$.  

Using the global spacetime bound \eqref{S bdd}, a standard application of the Strichartz inequality also yields the global spacetime bounds
\begin{align}\label{S bdd 2}
\|u\|_{L_t^2L_x^6(\R\times\R^3)} + \|u\|_{L_t^{\frac{28r}{3(r+6)}}L_x^{\frac{7r}{3(r-1)}}(\R\times\R^3)}\leq C(\|u_0\|_{L^2_x}).
\end{align}
Note that $(\frac{28r}{3(r+6)}, \frac{7r}{3(r-1)})$ is a Schr\"odinger-admissible pair for $2<r<6$.

As previously, we employ a small parameter $\eta>0$ that will be chosen later depending only on absolute constants (such as those in the dispersive and Strichartz estimates) and $\|u_0\|_{L^2_x}$.  We decompose $[0, \infty)$ into $J = J(\|u_0\|_{L^2_x}, \eta)$ many intervals $I_j = [T_{j-1}, T_j)$ such that 
\begin{align}\label{small norm}
 \|u\|_{L_t^2L_x^6(I_j\times\R^3)}<\eta \quad \qtq{and} \quad \|u\|_{L_t^{\frac{28r}{3(r+6)}}L_x^{\frac{7r}{3(r-1)}}(I_j\times\R^3)}<\eta.
\end{align}

We will show that for each $1\leq j\leq J$, we have
\begin{align}\label{X 3d bdd}
X(T_j) \lesssim \|u_0\|_{L_x^{r'}} + C(\|u_0\|_{L^2_x}) \bigl[X(T_{j-1}) +\eta^{c(r)} X(T_j)\bigr],
\end{align}
where $c(r)$ is a positive constant depending on $r$.  Choosing $\eta$ sufficiently small to defeat the absolute implicit constant in \eqref{X 3d bdd} and $C(\|u_0\|_{L^2_x})$, we readily obtain
$$
X(\infty)\leq C(\|u_0\|_{L^2_x})\|u_0\|_{L_x^{r'}},
$$
which demonstrates \eqref{decay 3d}.

We turn to \eqref{X 3d bdd}.  Fix $1\leq j\leq J$. For $t\in(0, T_j)$, we decompose 
\begin{align}\label{3d decomp}
u(t) &= e^{it\Delta} u_0 - i\int_0^{\frac t2 \wedge A} e^{i(t-s)\Delta} F(s)\, ds - i\int_{\frac t2 \wedge A}^{\frac t2} e^{i(t-s)\Delta} F(s)\, ds\notag\\
&\quad -  i\int_{\frac t2}^{\frac t2\vee(t-B)} e^{i(t-s)\Delta} F(s)\, ds -  i\int_{\frac t2\vee(t-B)}^t e^{i(t-s)\Delta} F(s)\, ds,
\end{align}
where $F= \pm |u|^{\frac43}u$ and $A,B$ are positive constants to be chosen later.  

By the dispersive estimate \eqref{L disp est'}, the contribution of the \emph{first term} on RHS\eqref{3d decomp} is easily seen to be acceptable:
\begin{align}\label{3d 1}
\|e^{it\Delta} u_0\|_{L^r_x} \lesssim t^{-3(\frac12-\frac1r)} \| u_0\|_{L_x^{r'}}.
\end{align}

To estimate the contribution of the \emph{second term} on RHS\eqref{3d decomp} we use the first bound on RHS\eqref{improv bdd} together with the observation that $|t-s|\sim t$ whenever $s\in [0, \frac t2]$.  In this way, employing also the H\"older inequality in the variable $s$ and \eqref{S bdd 2}, we get
\begin{align}\label{3d 2}
\Bigl\| \int_0^{\frac t2 \wedge A} e^{i(t-s)\Delta} F(s)\, ds \Bigr\|_{L_x^r}
&\lesssim t^{-3(\frac12-\frac1r)} \|F\|_{L^1_s L_x^{\frac r{r-1}}([0,\frac t2 \wedge A]\times \R^3)}\notag\\
&\lesssim t^{-3(\frac12-\frac1r)} A^{\frac{3(r-2)}{4r}} \|u\|_{L_t^{\frac{28r}{3(r+6)}}L_x^{\frac{7r}{3(r-1)}}([0,\frac t2 \wedge A]\times\R^3)}^{\frac73}\notag\\
&\leq t^{-3(\frac12-\frac1r)} A^{\frac{3(r-2)}{4r}} C(\|u_0\|_{L^2_x}).
\end{align}
On the region of integration $[0, \frac t2 \wedge A]\cap [T_{j-1}, T_j]$, we may use \eqref{small norm} in place of \eqref{S bdd 2} to obtain
\begin{align}\label{3d 2'}
\Bigl\| \int_{[0, \frac t2 \wedge A]\cap [T_{j-1}, T_j]} e^{i(t-s)\Delta} F(s)\, ds \Bigr\|_{L_x^r}
\lesssim t^{-3(\frac12-\frac1r)} \eta^{\frac73} A^{\frac{3(r-2)}{4r}}.
\end{align}

The next term we consider is the \emph{fourth term} on RHS\eqref{3d decomp} whose treatment parallels that of the second term just discussed:
\begin{align}\label{3d 4}
\Bigl\|\int_{\frac t2}^{\frac t2\vee(t-B)} &e^{i(t-s)\Delta} F(s)\, ds \Bigr\|_{L_x^r}\notag\\
&\lesssim \bigl\| |t-s|^{-3(\frac12-\frac1r)} F(s) \bigr\|_{L^1_s L_x^{\frac r{r-1}}([\frac t2, \frac t2\vee(t-B)]\times \R^3)}\notag\\
&\lesssim  \bigl\| |t-s|^{-3(\frac12-\frac1r)} \bigr\|_{L_s^{\frac{4r}{3(r-2)}}([0, t-B])} \|u\|_{L_t^{\frac{28r}{3(r+6)}}L_x^{\frac{7r}{3(r-1)}}([\frac t2, \frac t2\vee(t-B)]\times\R^3)}^{\frac73} \notag\\
&\leq B^{-\frac{3(r-2)}{4r}} C(\|u_0\|_{L^2_x}).
\end{align}
On the region of integration $[\frac t2, \frac t2\vee(t-B)]\cap [T_{j-1}, T_j]$, we may use \eqref{small norm} in place of \eqref{S bdd 2} to obtain
\begin{align}\label{3d 4'}
\Bigl\|\int_{[\frac t2, \frac t2\vee(t-B)]\cap [T_{j-1}, T_j]} e^{i(t-s)\Delta} F(s)\, ds \Bigr\|_{L_x^r} \lesssim \eta^{\frac73} B^{-\frac{3(r-2)}{4r}}.
\end{align}

We now turn to the \emph{third term} on RHS\eqref{3d decomp}.  We will again use the observation that for $s\in [0, \frac t2]$ we have $|t-s|\sim t$.  We will further subdivide the domain of integration into $[\frac t2 \wedge A, \frac t2] \cap [0, T_{j-1}]$ and $[\frac t2 \wedge A, \frac t2] \cap [T_{j-1}, T_j]$ and will consider separately three different regions for the exponent $r$.\\[2mm]
\noindent \textbf{Case 1:}  $2<r<\frac83$.  With $I= [\frac t2 \wedge A, \frac t2] \cap [0, T_{j-1}]$, using the second bound on RHS\eqref{improv bdd}, we obtain
\begin{align*}
\Bigl\| \int_I e^{i(t-s)\Delta} F(s)\, ds \Bigr\|_{L_x^r} 
&\lesssim |t|^{-3(\frac12-\frac1r)} \Bigl\| \|u(s)\|_{L_x^{\frac{7r}{3r-1}}}^{\frac73} \Bigr\|_{L_s^{\frac r{r-1}}(I)}.
\end{align*}
For this region of $r$ and $s\in I$, we may bound
$$
\|u(s)\|_{L_x^{\frac{7r}{3r-1}}} \lesssim \|u(s)\|_{L_x^6}^{1-\theta} \|u(s)\|_{L_x^r}^\theta \lesssim \|u(s)\|_{L_x^6}^{1-\theta} |s|^{-3(\frac12-\frac1r)\theta} \|u\|_{X(T_{j-1})}^\theta
$$
with $\theta = \frac{11r-6}{7(6-r)}$ (and so $1-\theta=  \frac{6(8-3r)}{7(6-r)}$).  Thus, applying the H\"older inequality in the variable $s$ and \eqref{S bdd 2}, we get
\begin{align}\label{3d 3r1}
\Bigl\| \int_I & e^{i(t-s)\Delta} F(s)\, ds \Bigr\|_{L_x^r} \notag\\
&\lesssim |t|^{-3(\frac12-\frac1r)} \|u\|_{L_t^2L_x^6(I\times\R^3)}^{\frac{2(8-3r)}{6-r}}\|u\|_{X(T_{j-1})}^{\frac{11r-6}{3(6-r)}} \bigl\|s^{-\frac{(r-2)(11r-6)}{2r(6-r)}}\bigl\|_{L_s^{\frac{r(6-r)}{(r-2)(2r+3)}}(I)}\notag\\
&\leq |t|^{-3(\frac12-\frac1r)} A^{-\frac{(r-2)(7r-12)}{2r(6-r)}}\|u\|_{X(T_{j-1})}^{\frac{11r-6}{3(6-r)}} C(\|u_0\|_{L_x^2}).
\end{align}

If instead $I=[\frac t2 \wedge A, \frac t2] \cap [T_{j-1}, T_j]$, using \eqref{small norm} in place of \eqref{S bdd 2} we get
\begin{align}\label{3d 3r1'}
\Bigl\| \int_I  e^{i(t-s)\Delta} & F(s)\, ds \Bigr\|_{L_x^r} \lesssim |t|^{-3(\frac12-\frac1r)}  \eta^{\frac{2(8-3r)}{6-r}} A^{-\frac{(r-2)(7r-12)}{2r(6-r)}}\|u\|_{X(T_j)}^{\frac{11r-6}{3(6-r)}}.
\end{align}

\noindent \textbf{Case 2:}  $\frac83\leq r<4$.  With $I= [\frac t2 \wedge A, \frac t2] \cap [0, T_{j-1}]$, using the second bound on RHS\eqref{improv bdd}, we proceed as above to estimate
\begin{align*}
\Bigl\| \int_I e^{i(t-s)\Delta} F(s)\, ds \Bigr\|_{L_x^r} 
&\lesssim |t|^{-3(\frac12-\frac1r)} \Bigl\| \|u(s)\|_{L_x^{\frac{7r}{3r-1}}}^{\frac73} \Bigr\|_{L_s^{\frac r{r-1}}(I)} .
\end{align*}
However, for this region of $r$ and $s\in I$, we bound
$$
\|u(s)\|_{L_x^{\frac{7r}{3r-1}}} \lesssim \|u(s)\|_{L_x^2}^{1-\theta} \|u(s)\|_{L_x^r}^\theta \lesssim \|u(s)\|_{L_x^2}^{1-\theta} |s|^{-3(\frac12-\frac1r)\theta} \|u\|_{X(T_{j-1})}^\theta
$$
with $\theta = \frac{r+2}{7(r-2)}$ (and so $1-\theta=  \frac{2(3r-8)}{7(r-2)}$).  Thus, applying the H\"older inequality in the variable $s$, we get
\begin{align}\label{3d 3r2}
\Bigl\| \int_I  e^{i(t-s)\Delta} F(s)\, ds \Bigr\|_{L_x^r} 
&\lesssim |t|^{-3(\frac12-\frac1r)} \|u\|_{L_t^\infty L_x^2(I\times\R^3)}^{\frac{2(3r-8)}{3(r-2)}}\|u\|_{X(T_{j-1})}^{\frac{r+2}{3(r-2)}} \bigl\|s^{-\frac{r+2}{2r}}\bigl\|_{L_s^{\frac{r}{r-1}}(I)}\notag\\
&\leq |t|^{-3(\frac12-\frac1r)} A^{-\frac{4-r}{2r}}\|u\|_{X(T_{j-1})}^{\frac{r+2}{3(r-2)}}  C(\|u_0\|_{L_x^2}).
\end{align}

If instead $I=[\frac t2 \wedge A, \frac t2] \cap [T_{j-1}, T_j]$, the bound \eqref{3d 3r2} becomes
\begin{align}\label{3d 3r2'}
\Bigl\| \int_I & e^{i(t-s)\Delta} F(s)\, ds \Bigr\|_{L_x^r} \leq |t|^{-3(\frac12-\frac1r)} A^{-\frac{4-r}{2r}}\|u\|_{X(T_j)}^{\frac{r+2}{3(r-2)}}  C(\|u_0\|_{L_x^2}).
\end{align}

\noindent \textbf{Case 3:}  $4\leq r<6$.  With $I= [\frac t2 \wedge A, \frac t2] \cap [0, T_{j-1}]$, we use the first bound on RHS\eqref{improv bdd} to obtain
\begin{align*}
\Bigl\| \int_I e^{i(t-s)\Delta} F(s)\, ds \Bigr\|_{L_x^r} 
&\lesssim |t|^{-3(\frac12-\frac1r)} \Bigl\| \|u(s)\|_{L_x^{\frac{7r}{3(r-1)}}}^{\frac73} \Bigr\|_{L_s^1(I)} .
\end{align*}
For this region of $r$ and $s\in I$, we bound
$$
\|u(s)\|_{L_x^{\frac{7r}{3(r-1)}}} \lesssim \|u(s)\|_{L_x^2}^{1-\theta} \|u(s)\|_{L_x^r}^\theta \lesssim \|u(s)\|_{L_x^2}^{1-\theta} |s|^{-3(\frac12-\frac1r)\theta} \|u\|_{X(T_{j-1})}^\theta
$$
with $\theta = \frac{r+6}{7(r-2)}$ (and so $1-\theta=  \frac{2(3r-10)}{7(r-2)}$).  Thus, applying the H\"older inequality in the variable $s$, we get
\begin{align}\label{3d 3r3}
\Bigl\| \int_I  e^{i(t-s)\Delta} F(s)\, ds \Bigr\|_{L_x^r} 
&\lesssim |t|^{-3(\frac12-\frac1r)} \|u\|_{L_t^\infty L_x^2(I\times\R^3)}^{\frac{2(3r-10)}{3(r-2)}}\|u\|_{X(T_{j-1})}^{\frac{r+6}{3(r-2)}} \bigl\|s^{-\frac{r+6}{2r}}\bigl\|_{L_s^1(I)}\notag\\
&\leq |t|^{-3(\frac12-\frac1r)} A^{-\frac{6-r}{2r}}\|u\|_{X(T_{j-1})}^{\frac{r+6}{3(r-2)}}  C(\|u_0\|_{L_x^2}).
\end{align}

If instead $I=[\frac t2 \wedge A, \frac t2] \cap [T_{j-1}, T_j]$, the bound \eqref{3d 3r3} becomes
\begin{align}\label{3d 3r3'}
\Bigl\| \int_I & e^{i(t-s)\Delta} F(s)\, ds \Bigr\|_{L_x^r} \leq |t|^{-3(\frac12-\frac1r)} A^{-\frac{6-r}{2r}}\|u\|_{X(T_j)}^{\frac{r+6}{3(r-2)}}  C(\|u_0\|_{L_x^2}).
\end{align}

We now optimize in $A$ the bounds we obtained for the second and third terms on RHS\eqref{3d decomp}.  Specifically, for $2<r<\frac83$ we optimize in $A$ between \eqref{3d 2} and \eqref{3d 2'} on one hand and \eqref{3d 3r1} and \eqref{3d 3r1'} on the other hand to obtain
\begin{align}\label{3d half time1}
\Bigl\| \int_0^{\frac t2} & e^{i(t-s)\Delta} F(s)\, ds \Bigr\|_{L_x^r} \leq |t|^{-3(\frac12-\frac1r)} C(\|u_0\|_{L_x^2}) \bigl[  \|u\|_{X(T_{j-1})} + \eta^{c(r)} \|u\|_{X(T_j)} \bigr]
\end{align}
for some positive constant $c(r)$ depending only on $r$.

For $\frac83\leq r<4$ we optimize in $A$ between \eqref{3d 2}--\eqref{3d 2'} and \eqref{3d 3r2}--\eqref{3d 3r2'}, while for $4\leq r<6$ we optimize between \eqref{3d 2}--\eqref{3d 2'} and \eqref{3d 3r3}--\eqref{3d 3r3'}.  In each case, we obtain a bound of the form \eqref{3d half time1} with a slightly different constant $c(r)$.

\medskip

Finally, we turn to the \emph{last term} on RHS\eqref{3d decomp}.  In this case, we observe that for $s\in [\frac t2, t]$ we have $s\sim t$.  We will further subdivide the domain of integration into $[\frac t2\vee(t-B), t] \cap [0, T_{j-1}]$ and $[\frac t2\vee(t-B), t] \cap [T_{j-1}, T_j]$ and will consider separately three different regions for the exponent $r$.\\[2mm]
\noindent \textbf{Case 1:}  $2<r<\frac83$.  With $I= [\frac t2\vee(t-B), t] \cap [0, T_{j-1}]$, we argue as for \eqref{3d 3r1} to obtain
\begin{align}\label{3d 4r1}
\Bigl\| \int_I & e^{i(t-s)\Delta} F(s)\, ds \Bigr\|_{L_x^r} 
\lesssim \Bigl\|  |t-s|^{-3(\frac12-\frac1r)}\|u(s)\|_{L_x^{\frac{7r}{3r-1}}}^{\frac73} \Bigr\|_{L_s^{\frac r{r-1}}(I)} \notag\\
&\lesssim \Bigl\|  |t-s|^{-3(\frac12-\frac1r)}\|u(s)\|_{L_x^6}^{\frac{2(8-3r)}{6-r}} \|u(s)\|_{L_x^r}^{\frac{11r-6}{3(6-r)}}  \Bigr\|_{L_s^{\frac r{r-1}}(I)} \notag\\
&\lesssim  |t|^{-\frac{(r-2)(11r-6)}{2r(6-r)}} \bigl\||t-s|^{-3(\frac12-\frac1r)}\bigr\|_{L^{\frac{r(6-r)}{(r-2)(2r+3)}}_s(I)}\|u\|_{L_t^2L_x^6(I\times\R^3)}^{\frac{2(8-3r)}{6-r}}\|u\|_{X(T_{j-1})}^{\frac{11r-6}{3(6-r)}}\notag\\
&\leq |t|^{-\frac{(r-2)(11r-6)}{2r(6-r)}} B^{\frac{(r-2)(7r-12)}{2r(6-r)}}\|u\|_{X(T_{j-1})}^{\frac{11r-6}{3(6-r)}} C(\|u_0\|_{L_x^2}).
\end{align}

If instead $I=[\frac t2\vee(t-B), t] \cap [T_{j-1}, T_j]$, using \eqref{small norm} we get
\begin{align}\label{3d 4r1'}
\Bigl\| \int_I  e^{i(t-s)\Delta} & F(s)\, ds \Bigr\|_{L_x^r} \lesssim |t|^{-\frac{(r-2)(11r-6)}{2r(6-r)}}  \eta^{\frac{2(8-3r)}{6-r}}  B^{\frac{(r-2)(7r-12)}{2r(6-r)}}\|u\|_{X(T_j)}^{\frac{11r-6}{3(6-r)}}.
\end{align}

\noindent \textbf{Case 2:}  $\frac83\leq r<4$.  With $I= [\frac t2\vee(t-B), t]  \cap [0, T_{j-1}]$, we argue as for \eqref{3d 3r2} to estimate
\begin{align}\label{3d 4r2}
\Bigl\| \int_I e^{i(t-s)\Delta} F(s)\, ds \Bigr\|_{L_x^r} 
&\lesssim \Bigl\|  |t-s|^{-3(\frac12-\frac1r)}\|u(s)\|_{L_x^{\frac{7r}{3r-1}}}^{\frac73} \Bigr\|_{L_s^{\frac r{r-1}}(I)} \notag\\
&\lesssim \Bigl\|  |t-s|^{-3(\frac12-\frac1r)}\|u(s)\|_{L_x^2}^{\frac{2(3r-8)}{3(r-2)}} \|u(s)\|_{L_x^r}^{\frac{r+2}{3(r-2)}}  \Bigr\|_{L_s^{\frac r{r-1}}(I)} \notag\\
&\lesssim  |t|^{-\frac{r+2}{2r}} \bigl\||t-s|^{-3(\frac12-\frac1r)}\bigr\|_{L^{\frac{r}{r-1}}_s(I)}\|u\|_{L_t^\infty L_x^2(I\times\R^3)}^{\frac{2(3r-8)}{3(r-2)}}\|u\|_{X(T_{j-1})}^{\frac{r+2}{3(r-2)}}\notag\\
&\leq  |t|^{-\frac{r+2}{2r}} B^{\frac{4-r}{2r}}\|u\|_{X(T_{j-1})}^{\frac{r+2}{3(r-2)}} C(\|u_0\|_{L_x^2}).
\end{align}

If instead $I=[\frac t2\vee(t-B), t]  \cap [T_{j-1}, T_j]$, the bound \eqref{3d 3r2} becomes
\begin{align}\label{3d 4r2'}
\Bigl\| \int_I & e^{i(t-s)\Delta} F(s)\, ds \Bigr\|_{L_x^r} \leq |t|^{-\frac{r+2}{2r}} B^{\frac{4-r}{2r}}\|u\|_{X(T_j)}^{\frac{r+2}{3(r-2)}}  C(\|u_0\|_{L_x^2}).
\end{align}

\noindent \textbf{Case 3:}  $4\leq r<6$.  With $I= [\frac t2\vee(t-B), t]  \cap [0, T_{j-1}]$, we argue as for \eqref{3d 3r3} to obtain
\begin{align}\label{3d 4r3}
\Bigl\| \int_I e^{i(t-s)\Delta} F(s)\, ds \Bigr\|_{L_x^r} 
&\lesssim \Bigl\|  |t-s|^{-3(\frac12-\frac1r)}\|u(s)\|_{L_x^{\frac{7r}{3(r-1)}}}^{\frac73} \Bigr\|_{L_s^1(I)} \notag\\
&\lesssim \Bigl\|  |t-s|^{-3(\frac12-\frac1r)}\|u(s)\|_{L_x^2}^{\frac{2(3r-10)}{3(r-2)}} \|u(s)\|_{L_x^r}^{\frac{r+6}{3(r-2)}}  \Bigr\|_{L_s^1(I)} \notag\\
&\lesssim  |t|^{-\frac{r+6}{2r}} \bigl\||t-s|^{-3(\frac12-\frac1r)}\bigr\|_{L^1_s(I)}\|u\|_{L_t^\infty L_x^2(I\times\R^3)}^{\frac{2(3r-10)}{3(r-2)}}\|u\|_{X(T_{j-1})}^{\frac{r+6}{3(r-2)}}\notag\\
&\leq  |t|^{-\frac{r+6}{2r}} B^{\frac{6-r}{2r}}\|u\|_{X(T_{j-1})}^{\frac{r+6}{3(r-2)}} C(\|u_0\|_{L_x^2}).
\end{align}

If instead $I=[\frac t2\vee(t-B), t]  \cap [T_{j-1}, T_j]$, the bound \eqref{3d 4r3} becomes
\begin{align}\label{3d 4r3'}
\Bigl\| \int_I & e^{i(t-s)\Delta} F(s)\, ds \Bigr\|_{L_x^r} \leq |t|^{-\frac{r+6}{2r}} B^{\frac{6-r}{2r}}\|u\|_{X(T_j)}^{\frac{r+6}{3(r-2)}}  C(\|u_0\|_{L_x^2}).
\end{align}

We now optimize in $B$ the bounds we obtained for the last two terms on RHS\eqref{3d decomp}.  Specifically, for $2<r<\frac83$ we optimize in $B$ between \eqref{3d 4}--\eqref{3d 4'} and \eqref{3d 4r1}--\eqref{3d 4r1'}, for $\frac83\leq r<4$ we optimize between \eqref{3d 4}--\eqref{3d 4'} and \eqref{3d 4r2}--\eqref{3d 4r2'}, while for $4\leq r<6$ we optimize between \eqref{3d 4}--\eqref{3d 4'} and \eqref{3d 4r3}--\eqref{3d 4r3'}.  In each case, we obtain
\begin{align}\label{3d half time2}
\Bigl\| \int_{\frac t2}^t e^{i(t-s)\Delta} F(s)\, ds \Bigr\|_{L_x^r} \leq |t|^{-3(\frac12-\frac1r)} C(\|u_0\|_{L_x^2}) \bigl[\|u\|_{X(T_{j-1})} + \eta^{c(r)}\|u\|_{X(T_j)} \bigr]
\end{align}
for some positive constant $c(r)$ depending only on $r$.

Combining \eqref{3d decomp}, \eqref{3d 1}, \eqref{3d half time1}, and \eqref{3d half time2}, we arrive at
$$
|t|^{3(\frac12-\frac1r)} \|u(t)\|_{L_x^r} \lesssim \|u_0\|_{L_x^{r'}} +  C(\|u_0\|_{L_x^2}) \bigl[\|u\|_{X(T_{j-1})} + \eta^{c(r)}\|u\|_{X(T_j)} \bigr],
$$
uniformly for $t\in [0, T_j]$.  This proves \eqref{X 3d bdd}, thereby completing the proof of Theorem~\ref{T:3d}. \qed


\end{document}